\documentclass[11pt,a4paper,english]{article}

\newcommand{\beq}{\begin{equation}}
\newcommand{\eeq}{\end{equation}}

\usepackage[latin1]{inputenc}
\usepackage{babel}
\usepackage[centertags]{amsmath}
\usepackage{amsfonts}
\usepackage{amssymb}
\usepackage{color}
\usepackage{amsthm}
\usepackage{newlfont}
\usepackage{graphicx}
\usepackage{dsfont}

\DeclareMathOperator*{\essinf}{ess\,inf}

\newtheorem{theorem}{Theorem}[section]
\newtheorem{lemma}[theorem]{Lemma}

\newtheorem{proposition}[theorem]{Proposition}
\newtheorem{definition}[theorem]{Definition}
\newtheorem{remark}[theorem]{Remark}
\newtheorem{Assumptions}[theorem]{Assumption}

\makeatletter  
\@addtoreset{equation}{section}
\def\theequation{\arabic{section}.\arabic{equation}}
\makeatother  
\begin{document}
\title{\textbf{Identifying the Free Boundary of a \\ Stochastic, Irreversible Investment Problem\\ via the Bank-El Karoui Representation Theorem}\footnote{These results extend a portion of the second author Ph.D.\ dissertation \cite{tesimia} under the supervision of the first author. This paper has been presented by the second author at several conferences thanks to the financial support by the German Research Foundation (DFG) via grant Ri 1128-4-1.}}
\author{Maria B. Chiarolla\thanks{Dipartimento di Metodi e Modelli per l'Economia, il Territorio e la Finanza, Universit\`{a} di Roma `La Sapienza', via del Castro Laurenziano 9, 00161 Roma, Italy; \texttt{maria.chiarolla@uniroma1.it}}
\and Giorgio Ferrari\thanks{Corresponding author.\ Center for Mathematical Economics, Bielefeld University, Universitaetsstrasse 25, D-33615 Bielefeld, Germany; \texttt{giorgio.ferrari@uni-bielefeld.de}}}
\date{\today}
\maketitle

\vspace{0.5cm}

{\textbf{Abstract.}} We study a stochastic, continuous time model on a finite horizon for a firm that produces a single good. We model the production capacity as an It\^o diffusion controlled by a nondecreasing process representing the cumulative investment. The firm aims to maximize its expected total net profit by choosing the optimal investment process. That is a singular stochastic control problem.
We derive some first order conditions for optimality and we characterize the optimal solution in terms of the \textsl{base capacity} process $l^{*}(t)$, i.e.\ the unique solution of a representation problem in the spirit of Bank and El Karoui \cite{BankElKaroui}.
We show that the base capacity is deterministic and it is identified with the free boundary $\hat{y}(t)$ of the associated optimal stopping problem, when the coefficients of the controlled diffusion are deterministic functions of time. This is a novelty in the literature on finite horizon singular stochastic control problems. As a subproduct this result allows us to obtain an integral equation for the free boundary, which we explicitly solve in the infinite horizon case for a Cobb-Douglas production function and constant coefficients in the controlled capacity process.
\smallskip

{\textbf{Key words}}:
irreversible investment, singular stochastic control, optimal stopping, free boundary, Bank and El Karoui's Representation Theorem, base capacity.

\smallskip

{\textbf{MSC2010 subsject classification}}: 91B70, 93E20, 60G40, 60H25.

\smallskip

{\textbf{JEL classification}}: C02, E22, D92, G31.

\section{Introduction}
\label{introduction}

We study a continuous time, singular stochastic investment problem on a finite horizon $T$ for a firm that produces a single good. The setting is as in Chiarolla and Haussmann \cite{Chiarolla2} but without leisure, wages and scrap value. The capacity process is a diffusion controlled by a nondecreasing process $\nu(t)$ representing the cumulative investment, i.e.
\begin{equation*}
\label{capacityintr}
\left\{
\begin{array}{ll}
dC^{y,\nu}(t)= C^{y,\nu}(t)[ -\mu_C(t) dt + \sigma_C(t) dW(t)] + f_C(t)d\nu(t),\,\,\,\,\,\,\,t\in[0,T), \\ \\
C^{y,\nu}(0)=y > 0.
\end{array}
\right.
\end{equation*}
The optimal investment problem is
\begin{equation}
\label{netprofitintro}
\sup_{\nu}\mathbb{E}\bigg\{\int_0^{T} e^{-\int_0^t \mu_F(s) ds}\,R(C^{y,\nu}(t))dt - \int_{[0,T)} e^{-\int_0^t \mu_F(s)ds} d\nu(t) \bigg\}.
\end{equation}
In \cite{Chiarolla2} the Authors proved the existence of the optimal investment process $\hat{\nu}$. As expected, the optimal time to invest $\tau^{*}$ was the solution of the associated optimal stopping problem.
In particular, under constant coefficients and a Cobb-Douglas production function, they obtained a variational formulation for the optimal stopping problem, i.e.\ a free boundary problem.
In order to characterize the moving boundary $\hat{y}(t)$ through an integral equation, the Authors proved the left-continuity of $\hat{y}(t)$ and assumed its right-continuity (cf.\ \cite{Chiarolla2}, Assumption-[Cfb]) since continuity of the free boundary was needed to prove the smooth fit property.

In this paper, rather than trying to generalize the variational approach to the case of time-dependent coefficients, we introduce a new approach based on a stochastic generalization of the Kuhn-Tucker conditions and we identify the free boundary by exploiting the Bank-El Karoui Representation Theorem (cf.\ \cite{BankElKaroui}, Theorems $1$ and $3$). As a subproduct, we obtain an integral equation for the free boundary which does not require a priori continuity of the free boundary and the smooth fit property to be derived.

The Bank-El Karoui Representation Theorem allows to write an optional process $X=\{X(t), t \in [0,T]\}$ such that $X(T)=0$ as an optional projection of the form
\beq
\label{bankelkarouirepresentation}
X(t) = \mathbb{E}\bigg\{\int_{(t,T]} f\Big(s,\sup_{t \leq v < s}\xi(v)\Big)\,\mu(ds)\,\Big|\mathcal{F}_t\bigg\},\,\,\,\,\,\,\,\,\,\,t \in [0,T],
\eeq
where $f=f(t,\xi)$ is a prescribed function strictly decreasing in $\xi$, $\mu$ a nonnegative optional random measure and $\{\xi(t), t \in [0,T]\}$ is a progressively measurable process to be found.
It was shown in \cite{BankElKaroui} that the representation problem (\ref{bankelkarouirepresentation}) is closely linked to the solution of stochastic optimization problems as continuous time dynamic allocation problems with a limited amount of effort to spend on a fixed number of projects (e.g., cf.\ \cite{ElKarouiKaratzasGittins}), or the optimal consumption choice problem in a general semimartingale setting with Hindy-Huang-Kreps utility functional (cf. \cite{BankRiedel1}).

It turns out that the representation problem (\ref{bankelkarouirepresentation}) allows to handle problems in a framework not necessarily Markovian, as optimal control problems with a time-dependent, stochastic fuel constraint (cf.\ \cite{Bank}, \cite{CFR}), or irreversible investment problems (cf.\ \cite{RiedelSu}) with deterministic capacity and with profit rate influenced by a stochastic parameter process, not necessarily a diffusion. In \cite{RiedelSu} \textsl{invest just enough to keep the production capacity above a certain lower bound} (`the base capacity') was shown to be the optimal investment strategy.
Clearly such a policy acts like the optimal control of singular stochastic control problems as the original Monotone Follower Problem (e.g., cf.\ \cite{Karatzas81} and \cite{KaratzasShreve84}) or, more generally, irreversible investment problems (cf.\ \cite{KaratzasBaldursson}, \cite{Chiarolla2}, \cite{Kobila}, \cite{AOksendal} and \cite{Wang}, among others. See also the Introduction of \cite{Chiarolla4} and \cite{Chiarolla2} for an extensive review on such subject). Therefore the `base capacity' and the free boundary arising in singular stochastic control problems must be linked.
Here we aim to prove such a link by identifying the `base capacity' $l^{*}(t)$ of our irreversible investment problem.

We start by proving some first order conditions for optimality. Then we obtain $l^{*}(t)$ as the unique solution of a representation problem \`{a} la Bank and El Karoui \cite{BankElKaroui} and we characterize the optimal solution of the investment problem in terms of $l^{*}(t)$ by the first order conditions for optimality.
In particular, we prove that the `base capacity' $l^{*}(t)$ is deterministic and it coincides with the free boundary $\hat{y}(t)$ of the original irreversible investment problem when the coefficients of the controlled diffusion and the manager's discount factor are deterministic. This is a novelty in the literature on finite horizon singular stochastic control problems. To the best of our knowledge, the connection between the optional solution of the Bank-El Karoui representation problem and the free boundary of optimal stopping problems has not received significant attention so far. In the infinite horizon case, a somehow related paper is \cite{BankBaumgarten}, in which the Authors study the optimal stopping problem of a one-dimensional diffusion $X$ when the reward function $u(x)-k$ depends on the parameter $k$. They construct a function $\gamma(x)$ (which they call a \textsl{universal stopping signal}) such that for each $k$ the optimal stopping region $\Gamma_k$ may be written as $\Gamma_k:=\{x: \gamma(x) \geq k\}$. Then, in the case of constant discount rate and under some additional uniform integrability assumptions on $X$, they prove that $\gamma(X(t))$ solves a suitable representation problem \`{a} la Bank-El Karoui. Since the optimal stopping region is written in terms of $\gamma$, one would expect the signal function $\gamma$ to be related in some way to the locus of single points representing each the free boundary of the $k$-th problem, however such relation is not clear nor explicit. In our case we deal with a finite time horizon problem and our setting is quite different from that of \cite{BankBaumgarten}. Under Markovian assumptions the solution of the Bank-El Karoui representation problem is deterministic and it coincides exactly with the free boundary of our singular control problem (\ref{netprofitintro}).
As a subproduct the representation problem for $l^{*}(t)$ provides an integral equation for the free boundary, which, once discretized, might be solved numerically by backward induction.

Notice that when $T=+\infty$ we are able to find the explicit form of the free boundary which we show to coincide with that obtained in \cite{Pham} by H. Pham via a viscosity solution approach.

The paper is organized as follows. In Section \ref{firmproblem} we introduce the optimal investment problem, whereas in Section \ref{FOCoptimality} we derive the first order conditions for optimality. In Section \ref{findingoptimalcapacity} we find the optimal production capacity. In Section \ref{FiniteHor}, under restrictions on the coefficients of the controlled diffusion, we show that $l^{*}(t)$ is deterministic and coincides with $\hat{y}(t)$. Section \ref{CDinfinitehorizon} is devoted to the analysis of the Cobb-Douglas case with infinite time horizon. In the Appendix we recall the variational approach of Chiarolla and Haussmann \cite{Chiarolla2} and we generalize some of their results to the case of deterministic, time-dependent coefficients. Such results are needed in Section  \ref{FiniteHor}.


\section{The Firm's Investment Problem}
\label{firmproblem}

The setting is as in Chiarolla and Haussmann \cite{Chiarolla2} but without leisure, wages and scrap value. We briefly recall their notation. An economy with finite horizon $T$ and productive sector represented by a firm is considered on a complete probability space $(\Omega, \mathcal{F},\mathbb{P})$ with filtration $\{\mathcal{F}_t, t \in [0,T]\}$. Such a filtration is the usual augmentation of the filtration generated by an exogeneous Brownian motion $\{W(t), t \in [0,T]\}$ and augmented by $\mathbb{P}$-null sets.
The firm produces at rate $R(C)$ when its capacity is $C$. The cumulative, irreversible investment up to time $t$ is denoted by $\nu(t)$. It is an a.s.\ finite, left-continuous with right-limits, nondecreasing, and adapted process. The irreversibility of investment is expressed by the nondecreasing nature of $\nu$. The production capacity $C^{y,\nu}$ associated to the investment strategy $\nu$ satisfies
\beq
\label{capacity}
\left\{
\begin{array}{ll}
dC^{y,\nu}(t)= C^{y,\nu}(t)[ -\mu_C(t) dt + \sigma_C(t) dW(t)] + f_C(t)d\nu(t),\quad\,\, t\in[0,T), \\ \\
C^{y,\nu}(0)=y > 0,
\end{array}
\right.
\eeq
where $\mu_C$, $\sigma_C$ and $f_C$ are given progressively measurable processes, uniformly bounded in $(\omega,t)$. Moreover $f_C$ is continuous with $0<k_f \leq f_C(t) \leq \kappa_f$ and $\mu_C \geq 0$.
Here $f_C$ is a conversion factor since any unit of investment is converted into $f_C$ units of production capacity.

By setting
\beq
\label{nubarradefinizione}
C^{0}(t) :=C^{1,0}(t),\,\,\,\,\,\,\,\,\,\,\,\,\,\,
\overline{\nu}(t) := \int_{[0,t)} \frac{f_C(s)}{C^{0}(s)}d\nu(s),
\eeq
we may write
\beq
\label{GBM}
C^{0}(t)=e^{-\int_0^t\mu_C(s)ds}\mathcal{M}_0(t),
\eeq
where the exponential martingale
\beq
\label{expmg}
\mathcal{M}_s(t):=e^{-\int_s^t\frac{1}{2}\sigma_C^2(u)du + \int_s^t \sigma_C(u)dW(u)},\,\,\,\,t \in [s,T],
\eeq
is defined for $s \in [0,T]$.
Without investment, $C^{0}$ represents the decay of a unit of initial capital and we have
\beq
\label{solutioncapacity}
C^{y,\nu}(t) = C^{0}(t)[y + \overline{\nu}(t)].
\eeq

The production function of the firm is a nonnegative, measurable function $R(C)$. We make the following
\begin{Assumptions}
\label{AssProfit}
the mapping $C \mapsto R(C)$ is strictly increasing and strictly concave with continuous derivative $R_{c}(C):=\frac{\partial}{\partial C}R(C)$ satisfying the Inada conditions $$\lim_{C \rightarrow 0}R_{c}(C)= \infty,\,\,\,\,\,\,\,\,\,\,\,\,\,\,\,\lim_{C \rightarrow \infty}R_{c}(C)= 0.$$
\end{Assumptions}
\noindent Our Assumption \ref{AssProfit} is not as general as the Assumption in \cite{Chiarolla2} but it is needed to apply the Bank-El Karoui Representation Theorem \cite{BankElKaroui}.

Each investment plan $\nu \in \mathcal{S}_o$ leads to the expected total profit net of investment
\beq
\label{netprofit}
\mathcal{J}_{0,y}(\nu)=\mathbb{E}\bigg\{\int_0^{T} e^{-\int_0^t \mu_F(s) ds}\,R(C^{y,\nu}(t))dt - \int_{[0,T)} e^{-\int_0^t \mu_F(s)ds} d\nu(t) \bigg\}
\eeq
where
\begin{eqnarray*}
\mathcal{S}_o\hspace{-0.25cm}&:=&\hspace{-0.25cm}\{\nu:\Omega \times [0,T] \mapsto  \mathbb{R}_{+}\,\,\mbox{nondecreasing,\,\,left-continuous,\,\,adapted\,\,s.t.}\,\, \nu(0)=0,\,\,\mathbb{P}\mbox{-a.s.}\}
\end{eqnarray*}
is the convex set of irreversible investment processes. Here $\mu_F$ is the firm's manager discount factor; it is a nonnegative, progressively measurable process, uniformly bounded in $(\omega,t)$.

\noindent The firm's problem is then
\beq
\label{optimalproblem}
V(0,y):=\sup_{\nu \in \mathcal{S}_o}\mathcal{J}_{0,y}(\nu),
\eeq
with $V$ finite thanks to Assumption \ref{AssProfit} (cf.\ \cite{Chiarolla2}, Proposition $2.1$).
Moreover, the strict concavity of $R$ and the affine nature of $C^{y,\nu}$ in $\nu$ imply that $\mathcal{J}_{0,y}(\nu)$ is strictly concave on $\mathcal{S}_o$. Hence if a solution $\hat{\nu}$ of (\ref{optimalproblem}) exists, it is unique.
The existence of the solution has been proved in \cite{Chiarolla2}, Theorem $3.1$. We provide a new characterization of it in Theorem \ref{ottimasol} below.


\section{First Order Conditions for Optimality}
\label{FOCoptimality}

As in \cite{BankRiedel1}, \cite{Bank}, \cite{RiedelSu} and \cite{Steg}, among others, we now aim to characterize the optimal solution of (\ref{optimalproblem}) by some first order conditions for optimality.

Let $\mathcal{T}$ denote the set of all stopping times with value in $[0,T]$, $\mathbb{P}$-a.s.
Note that the strictly concave functional $\mathcal{J}_{0,y}(\nu)$ admits the supergradient
\begin{eqnarray}
\label{gradient}
\nabla_{\nu}\mathcal{J}_{0,y}(\nu)(\tau)\hspace{-0.25cm}&:=&\hspace{-0.25cm}\mathbb{E}\bigg\{\,\int_{\tau}^T e^{-\int_0^s \mu_{F}(u) du} C^{0}(s)\frac{f_C(\tau)}{C^{0}(\tau)}\,R_{c}(C^{y,\nu}(s))\,ds\,\Big|\,\mathcal{F}_{\tau}\,\bigg\}  \\
&&\hspace{1.5cm} -\, e^{-\int_0^{\tau}\mu_F(u) du}\,\mathds{1}_{\{\tau< T\}}, \nonumber
\end{eqnarray}
for $\tau \in \mathcal{T}$.

\begin{remark}
The quantity $\nabla_{\nu}\mathcal{J}_{0,y}(\nu)(t)$ may be interpreted as the net marginal expected future profit resulting from an additional infinitesimal investment at time $t$.
Mathematically, $\nabla_{\nu}\mathcal{J}_{0,y}(\nu)$ can be viewed as the Riesz representation of the profit's gradient at $\nu$.
More precisely, we may define $\nabla_{\nu}\mathcal{J}_{0,y}(\nu)$ as the optional projection of the product-measurable process
\beq
\label{phiopzionale}
\phi(t):= \int_{t}^T e^{-\int_0^s \mu_F(u) du}C^{0}(s)\frac{f_C(t)}{C^{0}(t)}\,R_{c}(C^{y,\nu}(s))\,ds - e^{-\int_0^t \mu_F(u)du}\mathds{1}_{\{t< T\}},
\eeq
for $t \in [0,T]$.
Hence $\nabla_{\nu}\mathcal{J}_{0,y}(\nu)$ is uniquely determined up to $\mathbb{P}$-indistinguishability and it holds
\beq
\label{Jacodbrachetto}
\mathbb{E}\bigg\{\,\int_{[0,T)} \nabla_{\nu}\mathcal{J}_{0,y}(\nu)(t)d\nu(t)\bigg\} = \mathbb{E}\bigg\{\,\int_{[0,T)} \phi(t) d\nu(t)\bigg\}
\eeq
for all $\nu \in \mathcal{S}_o$ (cf.\ Theorem 1.33 in \cite{Jacod}).
\end{remark}
We shall prove that
\begin{theorem}
\label{optimalFOC}
Given problem (\ref{optimalproblem}), $\hat{\nu}(t)$ is optimal if and only if the following first-order conditions
\beq
\label{FOCoptimal1}
\nabla_{\nu}\mathcal{J}_{0,y}(\hat{\nu})(\tau)\leq 0,\,\,\,\,\forall \tau \in \mathcal{T},\,\,\,\mathbb{P}\mbox{-a.s.},
\eeq
\beq
\label{FOCoptimal2}
\mathbb{E}\bigg\{\int_{[0,T)}\nabla_{\nu}\mathcal{J}_{0,y}(\hat{\nu})(t)\,d\hat{\nu}(t)\bigg\} = 0,
\eeq
hold true.
\end{theorem}

\begin{proof}
We may start by proving the sufficient part. Let $\hat{\nu}$ satisfy the first-order conditions (\ref{FOCoptimal1}) and (\ref{FOCoptimal2}) and let $\nu \in \mathcal{S}_o$.
Then it follows from (\ref{solutioncapacity}) that
$$C^{y,\hat{\nu}}(t) - C^{y,\nu}(t) = \int_{[0,t)} C^{0}(t)\frac{f_C(s)}{C^{0}(s)}(d\hat{\nu}(s) -d\nu(s)).$$
Hence concavity of $R$ implies
\begin{eqnarray}
&& \mathcal{J}_{0,y}(\hat{\nu}) - \mathcal{J}_{0,y}(\nu) \nonumber \\
&& = \mathbb{E}\bigg\{\,\int_0^T e^{- \int_0^t \mu_F(u) du} \Big[R(C^{y,\hat{\nu}}(t)) - R(C^{y,\nu}(t))\Big]dt\, - \int_{[0,T)} e^{- \int_0^t \mu_F(u) du} (d\hat{\nu}(t) - d\nu(t))\bigg\}  \nonumber 
\end{eqnarray}
\begin{eqnarray}
\label{Sufficiency1}
&& = \mathbb{E}\bigg\{\,\int_0^T e^{- \int_0^t \mu_F(u) du} \Big[R(C^{y,\hat{\nu}}(t)) - R(C^{y,\nu}(t))\Big]dt\, - \int_{[0,T)} e^{- \int_0^t \mu_F(u) du} (d\hat{\nu}(t) - d\nu(t))\bigg\}  \nonumber \\
&& \geq \mathbb{E}\bigg\{\,\int_0^T e^{- \int_0^t \mu_F(u) du}\, R_c(C^{y,\hat{\nu}}(t)) (C^{y,\hat{\nu}}(t) - C^{y,\nu}(t))dt\, - \int_{[0,T)} e^{- \int_0^t \mu_F(u) du} (d\hat{\nu}(t) - d\nu(t))\bigg\} \nonumber \\
&& = \mathbb{E}\bigg\{\,\int_0^T e^{- \int_0^t \mu_F(u) du}\, R_c(C^{y,\hat{\nu}}(t)) \int_{[0,t)} C^{0}(t)\frac{f_C(s)}{C^{0}(s)} (d\hat{\nu}(s) -d\nu(s)) dt\,  \nonumber \\
&&\hspace{1.2cm} - \int_{[0,T)} e^{- \int_0^t \mu_F(u) du} (d\hat{\nu}(t) -d\nu(t))\bigg\}  \nonumber \\
&&= \mathbb{E}\bigg\{\,\int_{[0,T)} \bigg [ \int_t^T e^{- \int_0^s \mu_F(u) du} \,R_c(C^{y,\hat{\nu}}(s))\, C^{0}(s)\frac{f_C(t)}{C^{0}(t)}ds - e^{- \int_0^t \mu_F(u) du} \bigg] (d\hat{\nu}(t) -d\nu(t)) \bigg\} \nonumber \\
&&= \mathbb{E}\bigg\{\,\int_{[0,T)} \nabla_{\nu}\mathcal{J}_{0,y}(\hat{\nu})(t)\,(d\hat{\nu}(t) -d\nu(t)) \bigg\} \geq 0, \nonumber
\end{eqnarray}
where we have used Fubini's theorem in the third equality, and (\ref{Jacodbrachetto}), (\ref{FOCoptimal1}) and (\ref{FOCoptimal2}) in the last one. It follows that $\hat{\nu}$ is optimal for problem (\ref{optimalproblem}).

Necessity may be derived from \cite{Steg}, Proposition $3.2$, with $k(t):=f_C^{-1}(t)e^{-\int_0^t \mu_F(u) du}C^0(t)$ and $F(\omega,t,q):=e^{-\int_0^t \mu_F(\omega,u) du}R(qC^{0}(\omega,t))$, $\omega \in \Omega, t \in [0,T], q > 0$.
Although we do not have that the optional cost process $k$ is a supermartingale, however it is nonnegative and this is all that is needed in \cite{Steg}, proof of Proposition $3.2$, to get necessity of (\ref{FOCoptimal1}) and (\ref{FOCoptimal2}).
\end{proof}
\noindent Theorem (\ref{optimalFOC}) characterizes the optimal investment plan but it might not be useful if one aims to find the explicit solution, since the first order conditions are not always binding.

In what follows we construct the optimal capacity in terms of the `\textsl{base capacity}' $\{l^{*}(t), t \in [0,T]\}$ (cf.\ also \cite{RiedelSu}, Definition $3.1$) which represents the capacity level that is optimal for a firm starting at time $t$ with capacity zero. We show that it is optimal for (\ref{optimalproblem}) to invest up to the base capacity level if the current capacity level is below it; otherwise no investment is optimal. Mathematically, $l^{*}$ is the solution of the Bank-El Karoui representation problem \cite{BankElKaroui}.


\section{Finding the Optimal Capacity Process}
\label{findingoptimalcapacity}

The Bank-El Karoui Representation Theorem (cf.\ \cite{BankElKaroui}, Theorem 3 and Remark 2.1) states that, given
\begin{itemize}
 \item an optional process $X=\{X(t), t \in [0,T]\}$ of class (D), lower-semicontinuous in expectation with $X(T)=0$,
  \item a nonnegative optional random Borel measure $\mu(\omega,dt)$,
	\item $f(\omega,t,x): \Omega \times [0,T] \times \mathbb{R} \mapsto \mathbb{R}$ such that $f(\omega, t, \cdot): \mathbb{R} \mapsto \mathbb{R}$ is continuous, strictly decreasing from $+\infty$ to $-\infty$, and the stochastic process $f(\cdot, \cdot,x): \Omega \times [0,T] \mapsto \mathbb{R}$ is progressively measurable and integrable with respect to $d\mathbb{P} \otimes \mu(\omega,dt)$,
\end{itemize}
then there exists an optional process $\xi = \{\xi(t), t \in [0,T]\}$ taking values in $\mathbb{R} \cup \{-\infty\}$ such that for all $\tau \in \mathcal{T}$,
$$f(t,\sup_{\tau \leq u < t}\xi(u))\mathds{1}_{(\tau, T]}(t) \in \textbf{L}^1\left(d\mathbb{P}\otimes \mu(\omega,dt)\right)$$ and
\beq
\label{backwardgenerica}
\mathbb{E}\bigg\{\,\int_{(\tau, T]} \,f(s, \sup_{\tau \leq u < s} \xi(u))\,\mu(ds)\,\Big|\,\mathcal{F}_{\tau}\,\bigg\}= X(\tau).
\eeq
In \cite{BankElKaroui}, Lemma $4.1$ (see also \cite{BankFollmer}, Remark $1.4$-(ii)), a real valued process $\xi$ is considered upper right-continuous\footnote{Notice that usually the limit superior is defined as $$\limsup_{s \searrow t} \xi(s) := \lim_{\epsilon \downarrow 0} \sup_{s \in (t, (t+\epsilon) \wedge T)} \xi(s)$$ (see for example \cite{Yeh}, Chapter $3$, p.\ 250). Instead the definition in (\ref{urc}) is commonly referred to as the upper envelope $\hat{\xi}$ of $\xi$ and it is such that $\hat{\xi}(t) \geq \lim_{\epsilon \downarrow 0} \sup_{s \in (t, (t+\epsilon) \wedge T)} \xi(s)$ (see also \cite{Yeh}, Chapter $2$, Definition $7.24$ and Observation $7.25$, among others). In this paper we follow the point of view of the literature on the Bank-El Karoui Representation Theorem and its applications (like \cite{BankFollmer}, \cite{BankElKaroui} and \cite{BankKuchler}, among others) and hence we base our results on (\ref{urc}).} on $[0,T)$ in the sense of the French `Bourbakist' school (cf.,\ e.g., \cite{Bourbaki}, IV.24,  or \cite{Moreau}, Remark $3$ at pp.\ 29-30); that is, if, for each $t$, $\xi(t) = \limsup_{s \searrow t} \xi(s)$ with
\beq
\label{urc}
\limsup_{s \searrow t} \xi(s) := \lim_{\epsilon \downarrow 0} \sup_{s \in [t, (t+\epsilon) \wedge T]} \xi(s).
\eeq
Then, by \cite{BankElKaroui}, Theorem $1$, any progressively measurable, upper right-continuous solution $\xi$ to (\ref{backwardgenerica}) is uniquely determined up to optional sections on $[0,T)$ in the sense that
$$\xi(\tau) = \essinf_{\tau < \sigma \leq T}\Xi_{\tau,\sigma}, \qquad \tau \in [0,T),$$
where $\Xi_{\tau,\sigma}$ is the unique (up to a $\mathbb{P}$-null set) $\mathcal{F}_{\tau}$-measurable random variable satisfying
$$\mathbb{E}\{X(\tau) - X(\sigma) | \mathcal{F}_{\tau}\} = \mathbb{E}\bigg\{\int_{(\tau, \sigma]} f(t,\Xi_{\tau,\sigma})\,\mu(dt) \Big| \mathcal{F}_{\tau}\bigg\}.$$

\begin{lemma}
\label{existencexi}
There exists a unique optional, upper right-continuous, positive process $l^{*}(t)$ that solves
\begin{eqnarray}
\label{representationproblem}
&&\mathbb{E}\bigg\{\,\int_{\tau}^T e^{- \int_0^s \mu_F(u) du} C^{0}(s)\,R_{c}\bigg(C^{0}(s) \sup_{\tau \leq u < s}\bigg (\frac{l^{*}(u)}{C^{0}(u)}\bigg )\bigg)\,ds\,\Big|\,\mathcal{F}_{\tau}\,\bigg\} \nonumber \\
&&  = e^{-\int_0^{\tau}\mu_F(u)du}\frac{C^{0}(\tau)}{f_C(\tau)}\mathds{1}_{\{\tau< T\}}
\end{eqnarray}
for all $\tau \in \mathcal{T}$.
\end{lemma}
\begin{proof}

We apply the Bank-El Karoui Representation Theorem to
\beq
\label{identification}
X(\omega,t):= e^{-\int_0^{t}\mu_F(\omega,u)du}\frac{C^{0}(\omega,t)}{f_C(\omega,t)} \mathds{1}_{[0,T)}(t),\,\,\,\,\,\,\,\,\,\,\,\,\,\,
\mu(\omega,dt):= e^{- \int_0^t \mu_F(\omega,u) du}\,C^{0}(\omega,t)dt
\eeq
and
\beq
\label{identificationf}
f(\omega,t,x):=
\left\{
\begin{array}{ll}
\displaystyle R_c\left(-\frac{C^{0}(\omega,t)}{x}\right),\,\,\,\,\,\mbox{for}\,\,x<0,\\ \\
\displaystyle -x\,,\,\,\,\,\,\,\,\,\,\,\,\,\,\,\,\,\,\,\,\,\,\,\,\,\,\,\,\,\,\,\,\,\,\,\,\,\,\,\,\,\,\mbox{for}\,\,x\geq 0,
\end{array}
\right.
\eeq
and define
\begin{equation}
\label{Yxi1}
\Gamma^{\xi}(t):= \essinf_{t \leq \tau \leq T}\mathbb{E}\bigg\{\,\int_{t}^{\tau} f(u,\xi) \mu(du) + X(\tau)\,\Big|\mathcal{F}_{t}\bigg\},\qquad \xi \in \mathbb{R},\,\,\,t \in [0,T].
\end{equation}
Recall that $\Gamma^{\xi}$ of (\ref{Yxi1}) may be taken to be right-continuous and it is such that the mapping $\xi \mapsto \Gamma^{\xi}(\omega,t)$ is continuous and nonincreasing for $\omega \in \Omega$, $t \in [0,T]$ (cf.\ \cite{BankElKaroui}, Lemma $4.12$).

Then, the optional process (cf.\ \cite{BankElKaroui}, eq.\ (23) and Lemma 4.13)
\beq
\label{defxi1}
\xi^{*}(t):= \sup\Big\{ \xi \in \mathbb{R} : \Gamma^{\xi}(t) = X(t)\Big\},\,\,\,\,\,\,\,t \in [0,T),
\eeq
solves the representation problem
\beq
\label{representationproblem0}
e^{-\int_0^{\tau}\mu_F(u)du}\frac{C^{0}(\tau)}{f_C(\tau)}\mathds{1}_{\{\tau< T\}} = \mathbb{E}\bigg\{\,\int_{\tau}^T f(s,\sup_{\tau \leq u < s} \xi^{*}(u))\,\mu(ds)\,\Big|\,\mathcal{F}_{\tau}\,\bigg\}.
\eeq
We now claim (and we prove it below) that $\xi^{*}$ is upper right-continuous and a.s.\ strictly negative on $[0,T)$. Then, the upper right-continuous, strictly positive process
\beq
\label{deflstar}
l^{*}(t):= -\frac{C^{0}(t)}{\xi^{*}(t)}
\eeq
solves by (\ref{representationproblem0})
\begin{eqnarray*}
\label{representationproblem2}
e^{-\int_0^{\tau}\mu_F(u)du}\frac{C^{0}(\tau)}{f_C(\tau)} \mathds{1}_{\{\tau< T\}}& \hspace{-0.25cm} = \hspace{-0.25cm}&\mathbb{E}\bigg\{\,\int_{\tau}^T e^{- \int_0^s \mu_F(u) du} C^{0}(s)\,R_{c}\bigg(\frac{C^{0}(s)}{-\sup_{\tau \leq u < s}( - \frac{C^{0}(u)}{l^{*}(u)})}\bigg)\,ds\,\Big|\,\mathcal{F}_{\tau}\,\bigg\}\nonumber \\
&\hspace{-0.25cm} = \hspace{-0.25cm}& \mathbb{E}\bigg\{\,\int_{\tau}^T e^{- \int_0^s \mu_F(u) du} C^{0}(s)\,R_{c}\bigg(\frac{C^{0}(s)}{\inf_{\tau \leq u < s} (\frac{C^{0}(u)}{l^{*}(u)})}\bigg)\,ds\,\Big|\,\mathcal{F}_{\tau}\,\bigg\}\\
&\hspace{-0.25cm} = \hspace{-0.25cm}&\mathbb{E}\bigg\{\,\int_{\tau}^T e^{- \int_0^s \mu_F(u) du} C^{0}(s)\,R_{c}\bigg(C^{0}(s)\sup_{\tau \leq u < s} \bigg(\frac{l^{*}(u)}{C^{0}(u)}\bigg)\bigg)\,ds\,\Big|\,\mathcal{F}_{\tau}\,\bigg\},\nonumber
\end{eqnarray*}
i.e.\ (\ref{representationproblem}). Moreover, $\xi^{*}$ (and hence $l^{*}$) is unique up to optional sections by \cite{BankElKaroui}, Theorem $1$, as it is optional and upper right-continuous. Therefore it is unique up to indistinguishability by Meyer's optional section theorem (see, e.g., \cite{DellMeyer}, Theorem IV.86).

To complete the proof we must show that $\xi^{*}$ is indeed upper right-continuous and a.s.\ negative on $[0,T)$.
We start by proving its upper right-continuity. To accomplish that we only need to prove that $\xi^{*}$ has upper semi right-continuous sample paths, i.e.
\beq
\label{usrc}
\limsup_{s \searrow t}\xi^{*}(s) \leq \xi^{*}(t),
\eeq
since $$\limsup_{s \searrow t}\xi^{*}(s) \geq \xi^{*}(t)$$ by definition (cf.\ (\ref{urc})).
Thanks to \cite{DellLeng}, Proposition $2$, it suffices to show (cf.\ also \cite{BankKuchler}, proof of Theorem $1$)
\beq
\label{DellacherieLenglstat}
\lim_{n \rightarrow \infty} \xi^{*}(\tau_n) \leq \xi^{*}(\tau),
\eeq
for any $\tau \in [0,T)$ and any sequence of stopping times $\{\tau_n\}_{n \in \mathbb{N}}$ such that $\tau_n \downarrow \tau$ and for which there exists $\zeta:=\lim_{n \rightarrow \infty}\xi^{*}(\tau_n)$ a.s.
Then, fix $\epsilon > 0$ and use (\ref{defxi1}), right-continuity of $t \mapsto \Gamma^{\xi}(t)$ and continuity of $\xi \mapsto \Gamma^{\xi}(t)$ (see \cite{BankElKaroui}, Lemma $4.12$) to write
\beq
\label{secondofDL}
\Gamma^{\zeta - \epsilon}(\tau) = \lim_{n \rightarrow \infty} \Gamma^{\zeta - \epsilon}(\tau_n) = \lim_{n \rightarrow \infty} X(\tau_n) = X(\tau) = \Gamma^{\xi^{*}(\tau)}(\tau).
\eeq
It thus follows that $\zeta - \epsilon \leq \xi^{*}(\tau)$ for any $\epsilon > 0$, which implies (\ref{DellacherieLenglstat}); i.e., $\xi^{*}$ is upper right-continuous.

Finally, to prove that $\xi^{*}(t) < 0$ a.s.\ on $[0,T)$ define
$$\sigma := \inf \{t \in [0,T) : \xi^{*}(t) \geq 0 \} \wedge T,$$
then for $\omega \in \{\sigma < T\}$, upper right-continuity of $\xi^{*}$ implies $\xi^{*}(\sigma)\geq 0$ and therefore $\sup_{\sigma \leq u < s}\xi^{*}(u) \geq 0$ for all $s\in (\sigma, T]$. Hence, (\ref{representationproblem0}) with $\tau = \sigma$, i.e.
\beq
\label{xipos}
e^{-\int_0^{\sigma}\mu_F(u)du}\frac{C^{0}(\sigma)}{f_C(\sigma)}\mathds{1}_{\{\sigma< T\}} = - \mathbb{E}\bigg\{\,\int_{\sigma}^T e^{- \int_0^s \mu_F(u) du} C^{0}(s)\,\sup_{\sigma \leq u < s}\xi^{*}(u)\,ds\,\Big|\,\mathcal{F}_{\sigma}\,\bigg\},
\eeq
is not possible for $\omega \in \{\sigma < T\}$ since the right-hand side of (\ref{xipos}) is nonpositive, whereas the left-hand side is always strictly positive. It follows that $\sigma=T$ a.s.\ and hence $\xi^{*}(t) < 0$ for all $t \in [0,T)$ a.s.
\end{proof}

Notice that $l^{*}(t)$ may be found numerically by backward induction on a discretized version of problem (\ref{representationproblem}) (see \cite{BankFollmer}, Section $4$).
In some cases, when $T=+\infty$, (\ref{representationproblem}) has a closed form solution as in the case of a Cobb-Douglas production function (see Section \ref{CDinfinitehorizon} below).

We are now able to find the unique optimally controlled capacity plan for problem (\ref{optimalproblem}).

\begin{definition}
\label{basecapacity}
For a given positive process $l$, the capacity process that tracks $l$ is defined as
\beq
\label{Coptimal}
C^{(l)}(t): = C^{0}(t)\bigg( y \vee \sup_{0 \leq u < t}\bigg (\frac{l(u)}{C^{0}(u)}\bigg ) \bigg).
\eeq
\end{definition}

\begin{theorem}
\label{ottimasol}
Let $l^{*}(t)$ be the unique optional, upper right-continuous, positive solution of (\ref{representationproblem}) and let $C^{(l^{*})}$ be the capacity process that tracks $l^{*}$. Then the investment plan $\nu^{(l^{*})}$ that finances $C^{(l^{*})}$, i.e. $$d\nu^{(l^{*})}(t)= \frac{1}{f_C(t)}\,C^{(l^{*})}(t)[ \mu_C(t) dt - \sigma_C(t) dW(t)] + \frac{1}{f_C(t)}\,dC^{(l^{*})}(t),\,\,\,\,\,\,\mbox{with}\,\,\,\,\,\,\nu^{(l^{*})}(0)=0,$$ is optimal for the firm's problem (\ref{optimalproblem}).
\end{theorem}

\begin{proof}
In order to prove that $C^{(l^{*})}(t)$ is the optimal capacity, we only have to show that $C^{(l^{*})}(t)$ solves the two first-order conditions of Theorem \ref{optimalFOC}. In fact, for all $\tau \in \mathcal{T}$
\begin{eqnarray}
\label{ottimalitaC}
\lefteqn{\mathbb{E}\bigg\{\,\int_{\tau}^T e^{- \int_0^s \mu_F(u) du} C^{0}(s)\,R_{c}\big(C^{(l^{*})}(s)\big)\,ds\,\Big|\,\mathcal{F}_{\tau}\,\bigg\}} \nonumber \\
&&=\mathbb{E}\bigg\{\,\int_{\tau}^T e^{- \int_0^s \mu_F(u) du} C^{0}(s)\,R_{c}\bigg(C^{0}(s)\bigg( y \vee \sup_{0 \leq u < s}\bigg (\frac{l^{*}(u)}{C^{0}(u)}\bigg )\bigg)\bigg)\,ds\,\Big|\,\mathcal{F}_{\tau}\,\bigg\} \\
&&\leq \mathbb{E}\bigg\{\,\int_{\tau}^T e^{- \int_0^s \mu_F(u) du} C^{0}(s)\,R_{c}\bigg(C^{0}(s) \sup_{\tau \leq u < s}\bigg (\frac{l^{*}(u)}{C^{0}(u)}\bigg )\bigg)\,ds\,\Big|\,\mathcal{F}_{\tau}\,\bigg\} \nonumber\\
&&= e^{-\int_0^{\tau}\mu_F(u)du}\frac{C^{0}(\tau)}{f_C(\tau)}\mathds{1}_{\{\tau< T\}}, \nonumber
\end{eqnarray}
where in the last step we have used (\ref{representationproblem}). Notice that in (\ref{ottimalitaC}) we have equality if $\tau$ is a time of investment; that is a time of strict increase for $C^{(l^{*})}$, i.e.\ $dC^{(l^{*})}(\tau) > 0$. In fact, at such time, we have $C^{(l^{*})}(t) = C^{0}(t) \sup_{\tau \leq u < t}(\frac{l^{*}(u)}{C^{0}(u)})$ for $t \in (\tau,T]$. Hence (\ref{FOCoptimal1}) and (\ref{FOCoptimal2}) hold (see also (\ref{gradient})) and so $\nu^{(l^{*})}(t) \equiv \hat{\nu}(t)$.
\end{proof}

\begin{remark}
\label{trovonubarrato}
Recall that $C^{y,\hat{\nu}}(t)=  C^{0}(t)[y + \overline{\nu}^{y}(t)]$ (cf.\ (\ref{solutioncapacity})) where $\overline{\nu}^{y}(t):= \int_{[0,t)}\frac{f_C(s)}{C^0(s)}d\hat{\nu}(s)$. Hence it follows from (\ref{Coptimal}) with $l=l^{*}$ that
\beq
\label{nustargeneric}
\overline{\nu}^{y}(t) = \sup_{0 \leq u < t}\bigg( y \vee \frac{l^{*}(u)}{C^{0}(u)}\bigg) - y.
\eeq
Therefore
\beq
\label{nustargeneric2}
\overline{\nu}^{y}(t) = \sup_{0 \leq u < t}\bigg(\frac{l^{*}(u) - yC^{0}(u)}{C^{0}(u)}\bigg) \vee 0.
\eeq
\end{remark}


\section{Identifying the Base Capacity Process}
\label{FiniteHor}

In this Section we find the explicit link between our `base capacity' approach and the variational approach in Chiarolla and Haussmann \cite{Chiarolla2} based on the shadow value of installed capital, $v:=\frac{\partial}{\partial y}V$, with $V$ as in (\ref{optimalproblem}) (see Appendix \ref{solutionInvestmentProblem} for a generalization of \cite{Chiarolla2} in the case of deterministic, time-dependent coefficients).

We make the following

\begin{Assumptions}
\label{detrminsticcoeff}
$\mu_C(t)$, $\sigma_C(t)$, $f_C(t)$ and $\mu_F(t)$ are deterministic functions of $t \in [0,T]$.
\end{Assumptions}

Recall that (cf.\ also proof of Lemma \ref{existencexi}), if
\begin{eqnarray}
\label{Yxi}
\Gamma^{\xi}(t)\hspace{-0.25cm}&:=&\hspace{-0.25cm} \essinf_{t \leq \tau \leq T}\mathbb{E}\bigg\{\,\int_{t}^{\tau} e^{-\int_0^u \mu_F(r) dr} C^{0}(u)\,R_{c}\bigg( - \frac{1}{\xi}C^{0}(u) \bigg)\,du\, \nonumber \\
&& \hspace{3.2cm} +\, e^{-\int_0^{\tau}\mu_F(r) dr}C^{0}(\tau)\frac{1}{f_C(\tau)}\mathds{1}_{\{\tau < T\}}\,\Big|\mathcal{F}_{t}\bigg\},
\end{eqnarray}
for $\xi < 0$ and $t \in [0,T]$, then \cite{BankElKaroui}, Lemma $4.12$ and Lemma $4.13$, guarantee that
\begin{itemize}
	\item the stopping time
\beq
\label{taustar}
\tau^{\xi}(t):= \inf\bigg\{s \in [t,T) : \Gamma^{\xi}(s) = e^{-\int_0^s \mu_F(r)dr}C^{0}(s)\frac{1}{f_C(s)} \bigg\} \wedge T
\eeq
is optimal for (\ref{Yxi});
 \item the optional, upper right-continuous process
 \beq
 \label{defxi}
 \xi^{*}(t):= \sup\bigg\{ \xi < 0 : \Gamma^{\xi}(t) = e^{-\int_0^t \mu_F(r)dr}C^{0}(t)\frac{1}{f_C(t)}\bigg\},\,\,\,\,\,\,\,t \in [0,T),
 \eeq
uniquely solves the representation problem (\ref{representationproblem0}).
\end{itemize}

We now make an absolutely continuous change of probability measure. In fact, define the probability measure $\widetilde{\mathbb{P}}$ by $\widetilde{\mathbb{P}}(A):=\mathbb{E}\left\{\mathcal{M}_0(T) \mathds{1}_{A}\right\}$, for $A \in \mathcal{F}_{T}$, with $\mathcal{M}_0(T)$ as in (\ref{expmg}). Then the Radon-Nikodym derivative is
\beq
\label{RADON}
\frac{d\widetilde{\mathbb{P}}}{d\mathbb{P}}\Big|_{\mathcal{F}_t}=\mathcal{M}_0(t), \qquad t \in [0,T],
\eeq
and the process $\widetilde{W}(t):= W(t) - \int_0^t \sigma_C(u)du$, $t \in [0,T]$, is a standard Brownian motion under $\widetilde{\mathbb{P}}$.
We denote by $\widetilde{\mathbb{E}}\left\{\cdot\right\}$ the expectation w.r.t.\ $\widetilde{\mathbb{P}}$.

Hence, under $\widetilde{\mathbb{P}}$, by the continuous time Bayes' rule (see e.g.\ \cite{KaratzasShreve}) the process $e^{\int_0^t \mu_F(r)dr}\,\frac{\Gamma^{\xi}(t)}{C^{0}(t)}$ may be written as
\beq
\label{Ytilde}
\widetilde{\Gamma}^{\xi}(t):= \essinf_{t \leq \tau \leq T}\widetilde{\mathbb{E}}\bigg\{\,\int_{t}^{\tau} e^{-\int_{t}^u\overline{\mu}(r)dr}\,R_{c}\bigg( - \frac{1}{\xi}C^{0}(u) \bigg)\,du\, + e^{-\int_{t}^{\tau}\overline{\mu}(r)dr}\frac{1}{f_C(\tau)}\mathds{1}_{\{\tau < T\}}\Big|\,\mathcal{F}_{t}\,\bigg\},
\eeq
with $\overline{\mu}(t):=\mu_C(t) + \mu_F(t)$, and thus the optional process $\xi^{*}(t)$ of (\ref{defxi}) is
\beq
\label{taustar2}
\xi^{*}(t)= \sup\bigg\{\xi < 0 : \widetilde{\Gamma}^{\xi}(t) = \frac{1}{f_C(t)} \bigg\}, \quad t \in [0,T).
\eeq

For an appropriate value of $\xi$, we are now able to link $\widetilde{\Gamma}^{\xi}(t)$ to $v(t,y)$, the shadow value of installed capital defined in Appendix \ref{solutionInvestmentProblem} (cf.\ (\ref{vtyptilde})).
\begin{proposition}
\label{propYtildevty}
With $\widetilde{\Gamma}^{\xi}(t)$ as in (\ref{Ytilde}), $Y^{t,z}(u)=z\frac{C^0(u)}{C^0(t)} = z\widetilde{C}^t(u)$, $u \geq t$, as in (\ref{tildeC}),
\beq
\label{vtyptildebis}
v(t,z)=\inf_{t\leq \tau \leq T}\widetilde{\mathbb{E}}\bigg\{\,\int_t^{\tau} e^{-\int_t^u \overline{\mu}(r)dr}\,R_c\Big(Y^{t,z}(u)\Big)du + e^{-\int_t^{\tau}\overline{\mu}(r)dr}\frac{1}{f_C(\tau)}\mathds{1}_{\{\tau < T\}}\,\bigg\}
\eeq
as in (\ref{vtyptilde}), we have
\beq
\label{Ytildevty}
\widetilde{\Gamma}^{-\frac{1}{y}}(t) = v(t,yC^{0}(t)).
\eeq
\end{proposition}
\begin{proof}

The proof borrows arguments from \cite{Chiarolla4}, proof of Theorem $4.1$. For $t \in [0,T)$ and $\tau \in [t,T]$, notice that
\begin{eqnarray}
\label{tolgocondizionamento}
\lefteqn{\widetilde{\mathbb{E}}\bigg\{\,\int_{t}^{\tau} e^{-\int_{t}^u\overline{\mu}(r)dr}\,R_{c}\left(y C^{0}(u) \right)\,du\, + e^{-\int_{t}^{\tau}\overline{\mu}(r)dr}\frac{1}{f_C(\tau)}\mathds{1}_{\{\tau < T\}}\Big|\,\mathcal{F}_{t}\,\bigg\}}   \\
&&=\widetilde{\mathbb{E}}\bigg\{\,\int_{t}^{\tau} e^{-\int_{t}^u\overline{\mu}(r)dr}\,R_{c}\left( yC^{0}(t)\widetilde{C}^{t}(u)\right)\,du\, + e^{-\int_{t}^{\tau}\overline{\mu}(r)dr}\frac{1}{f_C(\tau)}\mathds{1}_{\{\tau < T\}}\Big|\,\mathcal{F}_{t}\,\bigg\}. \nonumber
\end{eqnarray}

In order to take care of the conditioning, it is convenient to work on the canonical probability space $\left(\overline{\Omega},\overline{\mathbb{P}}\right)$, where $\overline{\mathbb{P}}$ is the Wiener measure on $\overline{\Omega}:=\mathcal{C}_0\left([0,T]\right)$, the space of all continuous functions on $[0,T]$ which are zero at $t=0$.
We denote by $\widetilde{W}(t,\overline{\omega})=\overline{\omega}(t)$ the coordinate mapping on $\mathcal{C}_0\left([0,T]\right)$, with $\overline{\omega}=\left(\overline{\omega}_1,\overline{\omega}_2\right)$ where $\overline{\omega}_1=\left\{\widetilde{W}(v), 0 \leq v \leq t\right\}$ and $\overline{\omega}_2=\left\{\widetilde{W}(v)-\widetilde{W}(t), t \leq v \leq T\right\}=\left\{\widetilde{W}'(v), 0 \leq v \leq T-t\right\}$.
Now, independence of Brownian increments induces a product-measure on $\mathcal{C}_0\left([0,T]\right)=\mathcal{C}_0\left([0,t]\right)\times\mathcal{C}_0\left([0,T-t]\right)$ and $\tau$, with $\tau \geq t$ a.s., may be written in the form $\tau\left(\overline{\omega}_1,\overline{\omega}_2\right)=t + \tau'_{\overline{\omega}_1}\left(\overline{\omega}_2\right)$ with $\tau'_{\overline{\omega}_1}\left(\cdot\right)$ a $\left\{\mathcal{F}^{\widetilde{W}'}_{v}\right\}_{0\leq v \leq T-t}$-stopping time for every $\overline{\omega}_1 \in \overline{\Omega}$ (for a classical reference for this see \cite{DellMeyer}, Theorem $103$, p.\ $151$, among others).
Then, since $\widetilde{C}^{t}(\cdot)$ is independent of $\mathcal{F}_{t}$, if we denote by $\widetilde{\mathbb{E}}_{\,\overline{\omega}_2}\{\cdot\}$ the expectation over $\overline{\omega}_2$ or $W'$, we can write the last conditional expectation in (\ref{tolgocondizionamento}) as $\Phi_{\tau'_{\overline{\omega}_1}}(t, yC^0(t))$, where  
$$\Phi_{\tau'_{\overline{\omega}_1}}(t,z):= \widetilde{\mathbb{E}}_{\,\overline{\omega}_2}\bigg\{\,\int_{t}^{t + \tau'_{\overline{\omega}_1}} e^{-\int_{t}^u\overline{\mu}(r)dr}\,R_{c}\left( z\widetilde{C}^{t}(u)\right)\,du\, + e^{-\int_{t}^{t + \tau'_{\overline{\omega}_1}}\overline{\mu}(r)dr}\frac{1}{f_C(t + \tau'_{\overline{\omega}_1})}\mathds{1}_{\{t + \tau'_{\overline{\omega}_1}  < T\}}\,\bigg\},$$
for any $\overline{\omega}_1 \in \overline{\Omega}$ fixed. Now (\ref{Ytildevty}) follows from (\ref{Ytilde}) and (\ref{vtyptildebis}), since $\essinf_{t \leq \tau \leq T}\Phi_{\tau'_{\overline{\omega}_1}}(t, yC^0(t)) = \essinf_{t \leq t + \tau'_{\overline{\omega}_1} \leq T}\Phi_{\tau'_{\overline{\omega}_1}}(t, yC^0(t)) = v(t,yC^0(t))$, for every $\overline{\omega}_1 \in \overline{\Omega}$.
\end{proof}

Recall (\ref{taustar2}). The following Proposition provides another representation of the base capacity $l^{*}(t) := -\frac{C^{0}(t)}{\xi^{*}(t)}$ (cf. (\ref{deflstar})).
\begin{proposition}
\label{reprlstar}
The base capacity $l^{*}(t)$, unique optional, upper right-continuous, positive solution of (\ref{representationproblem}), admits the representation
\beq
\label{lstarrepresentation}
l^{*}(t)= \sup\bigg\{yC^{0}(t) > 0 : v(t,yC^{0}(t)) = \frac{1}{f_C(t)} \bigg\}, \quad t \in [0,T),
\eeq
with $v$ as in (\ref{vtyptildebis}).
\end{proposition}

\begin{proof}
For $t \in [0,T)$ and $y>0$ we have
\begin{eqnarray}
\label{lfrontiera}
l^{*}(t) &\hspace{-0.25cm}:=\hspace{-0.25cm}& - \frac{C^{0}(t)}{\xi^{*}(t)} = -\frac{C^{0}(t)}{\sup\left\{ \xi < 0 : \widetilde{\Gamma}^{\xi}(t) = \frac{1}{f_C(t)} \right\}}= -\frac{C^{0}(t)}{\sup\left\{ -\frac{1}{y} < 0 : \widetilde{\Gamma}^{-\frac{1}{y}}(t) = \frac{1}{f_C(t)} \right\}}\nonumber \\
&\hspace{-0.25cm}=\hspace{-0.25cm}&\frac{C^{0}(t)}{-\sup\left\{ -\frac{1}{y} < 0: \widetilde{\Gamma}^{-\frac{1}{y}}(t) = \frac{1}{f_C(t)} \right\}} = \frac{C^{0}(t)}{\inf \left\{\frac{1}{y} > 0 : \widetilde{\Gamma}^{-\frac{1}{y}}(t) = \frac{1}{f_C(t)} \right\}} \nonumber \\
&\hspace{-0.25cm}=\hspace{-0.25cm}&C^{0}(t)\,\sup \bigg\{y > 0 : \widetilde{\Gamma}^{-\frac{1}{y}}(t) = \frac{1}{f_C(t)} \bigg\} = \sup\bigg\{\,yC^{0}(t) > 0 : \widetilde{\Gamma}^{-\frac{1}{y}}(t) = \frac{1}{f_C(t)} \bigg\}\nonumber \\
&\hspace{-0.25cm}=\hspace{-0.25cm}&\sup\bigg\{\,yC^{0}(t) > 0 : v(t,yC^{0}(t)) = \frac{1}{f_C(t)} \bigg\},  \nonumber
\end{eqnarray}
where the last equality follows from Proposition \ref{propYtildevty}.
\end{proof}

Notice that $v(t,y) \leq \frac{1}{f_C(t)}$ for all $t \in [0,T)$ and $y>0$ (cf.\ (\ref{vtyptildebis})).
As in \cite{Chiarolla2}, eq.\ ($3.19$), introduce the \textsl{Continuation Region} (or `no-action region') of problem (\ref{vtyptildebis})
\beq
\label{contregion}
D:=\bigg\{(t,y) \in [0,T) \times (0,\infty) : v(t,y) < \frac{1}{f_C(t)} \bigg\}.
\eeq
Roughly speaking $D$ is the region where it is not profitable to invest, since the shadow value of installed capital is strictly less than the capital's replacement cost.
Similarly its complement is the \textsl{Stopping Region} (or `action region'), i.e.
\beq
\label{stoppingreg}
D^{c}:=\bigg\{(t,y) \in [0,T) \times (0,\infty) : v(t,y) = \frac{1}{f_C(t)} \bigg\}.
\eeq
That is the region where it is profitable to invest immediately.
The boundary between these two regions is the free boundary $\hat{y}(t)$ of the optimal stopping problem (\ref{vtyptildebis}).

\begin{theorem}
\label{lfreeboundary}
The base capacity process $l^{*}(t)$, unique optional, upper right-continuous, positive solution of (\ref{representationproblem}), is deterministic and coincides with the free boundary $\hat{y}(t)$ associated to the optimal stopping problem (\ref{vtyptildebis}). Hence
\beq
\label{lboundary}
l^{*}(t)=\sup\bigg\{\,z > 0: v(t,z) = \frac{1}{f_C(t)}\bigg\}\,\,\,\,\,\,\,\,\,\,\mbox{for}\,\,t \in [0,T).
\eeq
\end{theorem}
\begin{proof}
Recall (\ref{lstarrepresentation}). Fix $t \in [0,T)$ and set
$$\tilde{z}(\omega,y):=yC^{0}(\omega,t).$$
It follows that
\begin{eqnarray*}
\bigg\{\,yC^{0}(\omega,t) > 0 : v(t,yC^{0}(\omega,t))=\frac{1}{f_C(t)}\bigg\} &\hspace{-0.25cm}=\hspace{-0.25cm}&  \bigg\{\,\tilde{z}(\omega,y) > 0 : v(t,\tilde{z}(\omega,y)) = \frac{1}{f_C(t)} \bigg\} \nonumber \\
&\hspace{-0.25cm} \subseteq \hspace{-0.25cm}& \bigg\{\,z > 0 : v(t,z) = \frac{1}{f_C(t)} \bigg\} \nonumber
\end{eqnarray*}
for a.e.\ $\omega \in \Omega$ and $y >0$, hence the inclusion holds a.s.\ for all $y>0$.

The reverse inclusion follows from the particular dependence of the geometric Brownian motion on its initial value. In fact, if $z>0$, then for each $\omega \in \Omega$ and $t \in [0,T)$, $z$ may be written as 
$$z=\tilde{z}(\omega, y(\omega,z)),$$ with $y(\omega,z):=\frac{z}{C^{0}(\omega,t)}$.
Therefore
\begin{eqnarray*}
\bigg\{\,z > 0 : v(t,z) = \frac{1}{f_C(t)} \bigg\} &\hspace{-0.25cm} = \hspace{-0.25cm}& \bigg\{\,y(\omega,z)C^{0}(\omega,t) > 0 :
v(t,y(\omega,z)C^{0}(\omega,t)) = \frac{1}{f_C(t)}\bigg\}  \nonumber \\
&\hspace{-0.25cm} \subseteq \hspace{-0.25cm}& \bigg\{\,yC^{0}(\omega,t) > 0 : v(t,yC^{0}(\omega,t)) = \frac{1}{f_C(t)} \bigg\}. \nonumber
\end{eqnarray*}
This inclusion holds for a.e.\ $\omega \in \Omega$, thus a.s.\
Hence, it holds $\widetilde{\mathbb{P}}$-a.s.\ that
\beq
\label{identifico}
\sup\bigg\{\,yC^{0}(\omega,t) > 0 : v(t,yC^{0}(\omega,t)) = \frac{1}{f_C(t)} \bigg\} = \sup\bigg\{\,z > 0 : v(t,z) = \frac{1}{f_C(t)} \bigg\}
\eeq
and $l^{*}(t)$ is deterministic (cf.\ (\ref{lstarrepresentation})).
Now the right-hand side of (\ref{identifico}) (cf.\ \cite{Chiarolla2}, eq.\ ($3.13$)) identifies $l^{*}(t)$ with the free boundary $\hat{y}(t)$ of problem (\ref{vtyptildebis}).
\end{proof}

Since $\hat{y}(t)$ coincides with $l^{*}(t)$, equation (\ref{representationproblem}) provides an integral equation for the free boundary $\hat{y}(t)$ which does not require a priori continuity of $\hat{y}$ and the smooth fit property (as instead that in \cite{Chiarolla2}) to be derived.
\begin{theorem}
\label{eqintegralefb}
The free boundary $\hat{y}(t)$ of problem (\ref{vtyptildebis}) is the unique upper right-continuous, positive solution of the integral equation
\beq
\label{equazintegralefb}
\widetilde{\mathbb{E}}\bigg\{\int_{0}^{T-t} e^{-\int_{t}^{t+v}\overline{\mu}(r)dr}R_c\bigg(\sup_{0 \leq u' < v}\bigg(\hat{y}(t + u') \frac{C^{0}(t+v)}{C^{0}(t + u')}\bigg)\bigg)\,dv\bigg\} =\frac{1}{f_C(t)},\,\,\,\,\,t \in [0,T).
\eeq
\end{theorem}
\begin{proof}
Fix $t \in [0,T)$. Set $\tau = t$ and recall that $l^{*}(t) = \hat{y}(t)$. Then write (\ref{representationproblem}) under $\widetilde{\mathbb{P}}$ and apply the continuous time Bayes' Rule to obtain
$$\widetilde{\mathbb{E}}\bigg\{\int_{0}^{T-t} e^{-\int_{t}^{t+v}\overline{\mu}(r)dr}R_c\bigg(\sup_{0 \leq u' < v}\bigg(\hat{y}(t +u') \frac{C^{0}(t +v)}{C^{0}(t +u')}\bigg)\bigg)\,dv \, \Big|\,\mathcal{F}_t\bigg\} =\frac{1}{f_C(t)}.$$
Now (\ref{equazintegralefb}) follows since $\frac{C^{0}(t+v)}{C^{0}(t+u')}$, $v>u' \geq 0$, is independent of $\mathcal{F}_t$.
\end{proof}

As in \cite{Chiarolla2}, Section $4$, we now make the following
\begin{Assumptions}
\label{markovian}
\hspace{10cm}
\begin{enumerate}
\item $R(C)=\frac{1}{\alpha}C^{\alpha}$ with $\alpha \in (0,1)$ (i.e.\ Cobb-Douglas production function);
\item $\mu_C(t) \equiv \mu_C,\,\,\,\,\sigma_C(t) \equiv \sigma_C,\,\,\,\,\mu_F(t) \equiv \mu_F,\,\,\,\,f_C(t) \equiv f_C.$
\end{enumerate}
\end{Assumptions}

\begin{remark}
Notice that under the second part of Assumption \ref{markovian}, the process $\frac{C^{0}(t+v)}{C^{0}(t+u')}$ has the same law as $\frac{C^{0}(v)}{C^{0}(u')}$. Hence, the integral equation (\ref{equazintegralefb}) takes the form
\beq
\label{equazintegralefbmark}
\widetilde{\mathbb{E}}\bigg\{\int_{0}^{T-t} e^{-\overline{\mu}v}R_c\bigg(\sup_{0 \leq u' < v}\bigg(\hat{y}(t+u') \frac{C^{0}(v)}{C^{0}(u')}\bigg)\bigg)\,dv\bigg\} = \frac{1}{f_C}.
\eeq
\end{remark}
\noindent Under Assumption \ref{markovian}, the properties of the free boundary obtained in \cite{Chiarolla2} hold; in particular, $\hat{y}(t)$ is nonincreasing and left-continuous on $[0,T)$ (cf.\ \cite{Chiarolla2}, Proposition $4.3$ $[i]_{bdy}$).
At this point one could be tempted to conclude that it is also right-continuous since it coincides with the base capacity process $l^{*}(t)$. However that is not the case as $l^{*}(t)$ (according to definition (\ref{urc}) and footnote $1$) is not equal to what is usually called limit superior but it only coincides with its upper envelope, hence $l^{*}(t)$ is only greater or equal its limit superior (as commonly defined).

The identification of the free boundary $\hat{y}$ of problem (\ref{vtyptildebis}) with the base capacity process $l^{*}$ enables us to obtain an upper bound for $\hat{y}$ (see Proposition \ref{CobbDbound}), and its explicit form in the infinite horizon case when it reduces to a point (see Proposition \ref{aCD} below).

\begin{proposition}
\label{CobbDbound}
Under Assumption \ref{markovian} the boundary $\hat{y}(t)$ of the continuation region $D$ satisfies
\beq
\label{boundfb}
\hat{y}(t) \leq \bigg[f_C\bigg(\frac{1 - e^{-(\mu_F + \alpha\mu_C  + \frac{1}{2}\alpha(1 - \alpha)\sigma_C^2)(T-t)}}{ \mu_F + \alpha\mu_C  + \frac{1}{2}\alpha(1 - \alpha)\sigma_C^2}\bigg) \bigg]^{\frac{1}{1-\alpha}}=:y^{*}(t),
\eeq
for every $t \in [0,T)$.
\end{proposition}
\begin{proof}

Fix $t \in [0,T)$. The representation formula (\ref{representationproblem}) for $\tau=t$ and in the Cobb-Douglas case becomes
\beq
\label{CDTfinito}
e^{-\mu_F t}\frac{1}{f_C}= \mathbb{E}\bigg\{\,\int_{t}^{T} e^{-\mu_Fs} \frac{C^{0}(s)}{C^{0}(t)}\,\bigg(\sup_{t \leq u < s}\bigg ( C^{0}(s)\frac{l^{*}(u)}{C^{0}(u)}\bigg )\bigg)^{\alpha-1}\,ds\,\Big|\,\mathcal{F}_{t}\,\bigg\}.
\eeq
Set $\widetilde{\mu}_C:=\mu_C + \frac{1}{2}\sigma^2_C$, then the right-hand side of (\ref{CDTfinito}) gives
\begin{eqnarray}
\label{CDTfinito2}
\lefteqn{\mathbb{E}\bigg\{\,\int_{t}^{T} e^{-\mu_Fs} \frac{C^{0}(s)}{C^{0}(t)}\,\bigg(\sup_{t \leq u < s}\bigg ( C^{0}(s)\frac{l^{*}(u)}{C^{0}(u)}\bigg )\bigg)^{\alpha-1}\,ds\,\Big|\,\mathcal{F}_{t}\,\bigg\}} \nonumber \\
& & = \mathbb{E}\bigg\{\,\int_{t}^{T} e^{-\mu_Fs} \frac{C^{0}(s)}{C^{0}(t)}\,\inf_{t \leq u < s}\bigg ( C^{0}(s)\frac{l^{*}(u)}{C^{0}(u)}\bigg)^{\alpha-1}\,ds\,\Big|\,\mathcal{F}_{t}\,\bigg\} \\
& & \leq \mathbb{E}\bigg\{\int_{t}^{T} e^{-\mu_Fs} \,e^{-\widetilde{\mu}_C(s-t) + \sigma_C(W(s) -W(t))}\,[l^{*}(t)]^{\alpha-1}\,e^{(\alpha-1)\left(-\widetilde{\mu}_C(s -t) + \sigma_C(W(s) -W(t))\right)}\,ds\Big|\,\mathcal{F}_{t}\,\bigg\} \nonumber  \\
& & = e^{-\mu_F t}\,[l^{*}(t)]^{\alpha-1}\,\,\mathbb{E}\bigg\{\int_{0}^{T-t} e^{-\mu_Fv}\,e^{-\widetilde{\mu}_Cv + \sigma_C(W(v+t) -W(t))} \nonumber  \\
& & \hspace{4.5cm}\times\,e^{(\alpha-1)(-\widetilde{\mu}_Cv + \sigma_C(W(v+t) -W(t)))}\,dv\Big|\,\mathcal{F}_{t}\,\bigg\} \nonumber
\end{eqnarray}
Since the Brownian increments in the integral above are independent of $\mathcal{F}_{t}$, we obtain

\begin{eqnarray}
\label{CDTfinito2bis}
\lefteqn{\mathbb{E}\bigg\{\,\int_{t}^{T} e^{-\mu_Fs} \frac{C^{0}(s)}{C^{0}(t)}\,\bigg(\sup_{t \leq u < s}\bigg ( C^{0}(s)\frac{l^{*}(u)}{C^{0}(u)}\bigg )\bigg)^{\alpha-1}\,ds\,\Big|\,\mathcal{F}_{t}\,\bigg\}} \nonumber \\
&& \leq e^{-\mu_F t}\,[l^{*}(t)]^{\alpha-1}\,\mathbb{E}\bigg\{\int_{0}^{T-t} e^{-\mu_Fv}\,e^{-\widetilde{\mu}_Cv + \sigma_C(W(v+t) -W(t))} \nonumber \\
&& \hspace{4.5cm} \times\,e^{(\alpha-1)(-\widetilde{\mu}_C v + \sigma_C(W(v+t) -W(t)))}\,\bigg\}\,dv \\
&& = e^{-\mu_F t}\,[l^{*}(t)]^{\alpha-1}\,\int_{0}^{T-t} e^{-\mu_Fv} \,\mathbb{E}\left\{\,e^{\alpha\left(-\widetilde{\mu}_Cv + \sigma_C(W(v+t) -W(t))\right)}\right\}\,dv\nonumber \\
&&= e^{-\mu_F t}\,[l^{*}(t)]^{\alpha-1}\,\int_{0}^{T-t} e^{-\mu_Fv}e^{-\alpha \widetilde{\mu}_C v}\,e^{\frac{1}{2}\alpha^2\sigma_C^2\, v} \,dv. \nonumber
\end{eqnarray}

Notice that
$$\mu_F + \alpha\widetilde{\mu}_C - \frac{1}{2}\alpha^2\sigma_C^2 = \mu_F + \alpha\mu_C  + \frac{1}{2}\alpha(1 -\alpha)\sigma_C^2 > 0,$$
hence (\ref{CDTfinito}) and (\ref{CDTfinito2bis}) imply that
\begin{eqnarray}
\label{CDTfinito3}
e^{-\mu_F t}\frac{1}{f_C} & \hspace{-0.25cm} \leq  \hspace{-0.25cm}&  e^{-\mu_F t}\, [l^{*}(t)]^{\alpha-1} \int_{0}^{T-t} e^{-\left(\mu_F + \alpha\mu_C  + \frac{1}{2}\alpha(1 -\alpha)\sigma_C^2\right) v}\,dv  \\
& \hspace{-0.25cm} =  \hspace{-0.25cm}& e^{-\mu_F t}\, [l^{*}(t)]^{\alpha-1} \bigg(\frac{ 1 - e^{-\left(\mu_F +\alpha\mu_C  + \frac{1}{2}\alpha(1 -\alpha)\sigma_C^2\right)(T-t)}}{\mu_F +\alpha\mu_C  + \frac{1}{2}\alpha(1 -\alpha)\sigma_C^2}\bigg).\nonumber
\end{eqnarray}
Now (\ref{CDTfinito3}) gives
\beq
\label{CDTfinito4}
[l^{*}(t)]^{1 - \alpha} \leq f_C\,\bigg(\frac{ 1 - e^{-\left(\mu_F +\alpha\mu_C  + \frac{1}{2}\alpha(1 -\alpha)\sigma_C^2\right)(T-t)}}{\mu_F +\alpha\mu_C  + \frac{1}{2}\alpha(1 -\alpha)\sigma_C^2}\bigg) =:[y^{*}(t)]^{1-\alpha},
\eeq
and (\ref{boundfb}) follows from the identification of $l^{*}(\cdot)$ with $\hat{y}(\cdot)$ (cf.\ Theorem \ref{lfreeboundary}).
\end{proof}

\begin{remark}
Notice that the curve $y^{*}(t)$ is exactly what in \cite{Chiarolla4} was incorrectly identified as the free boundary between the `action' and the `no-action' regions. In \cite{Chiarolla2} the authors characterized the free boundary $\hat{y}(t)$ as the unique solution of a nonlinear integral equation (see \cite{Chiarolla2}, Theorem 4.8). Then, by using a discrete approximation of such integral equation, they showed that $\hat{y}(t) \leq y^{*}(t)$, for $t \leq T$. That is exactly what we proved here in Proposition \ref{CobbDbound}.
\end{remark}

\begin{remark}
The arguments in the proof of Proposition \ref{CobbDbound} apply even under the more general conditions of Assumption \ref{detrminsticcoeff}. That is, under deterministic, time-dependent coefficients we have
$$\hat{y}(t) \leq \bigg[ f_C(t) \int_0^{T-t} e^{-\int_t^{v+t}\left(\mu_F(s) + \alpha\mu_C(s) + \frac{1}{2}\alpha(1-\alpha)\sigma_C^2(s)\right)\,ds}dv\bigg]^{\frac{1}{1-\alpha}}, \qquad t \in [0,T).$$
\end{remark}

In this Section we have linked the Bank-El Karoui's probabilistic approach to the variational approach followed by Chiarolla and Haussmann in \cite{Chiarolla4} and \cite{Chiarolla2} for an irreversible investment problem similar to (\ref{optimalproblem}). Under Assumption \ref{detrminsticcoeff} we have proved that the base capacity process $l^{*}(t)$ is a deterministic process and it coincides with the free boundary of the optimal stopping problem (\ref{vtyptildebis}).
We have characterized the free boundary as the unique solution of an integral equation based on the stochastic Representation Theorem of \cite{BankElKaroui}.
Even under Assumption \ref{markovian}, the integral equation for the free boundary (\ref{equazintegralefb}) cannot be analitically solved when the time horizon is finite. However it is possible to find a curve
bounding the free boundary from above. In Section \ref{CDinfinitehorizon} we shall see that, instead, when $T=+\infty$ (as in H.\ Pham \cite{Pham}) the free boundary is a constant whose value we find explicitly by applying Proposition \ref{aCD}.


\section{Explicit Results when $T = + \infty$}
\label{CDinfinitehorizon}

In this Section, with $T=+\infty$ and under Assumption \ref{markovian}, we set $f_C = 1$ in order to compare our finding with the results in H.\ Pham \cite{Pham}. As one would expect, when the time horizon is infinite, the free boundary is a point. That is what we show below.
\begin{proposition}
\label{aCD}
The unique solution of the representation problem (\ref{representationproblem})
is given by
\beq
\label{lCobbDouglas}
l^{*}(t) = \Big[\frac{2}{2\mu_F - \sigma_C^2\beta_{-} - \alpha\sigma_C^2(1 + \beta_{+})}\Big]^{\frac{1}{1-\alpha}} = : a
\eeq
where $\beta_{\pm}$ are, respectively, the positive and negative roots of $\frac{1}{2}\sigma_C^2 x^2 + \widetilde{\mu}_Cx - \mu_F =0$ with $\widetilde{\mu}_C := \mu_C + \frac{1}{2}\sigma_C^2$.

Hence (cf.\ Definition \ref{basecapacity} and Theorem \ref{ottimasol}) the optimal capacity is given by
\beq
\label{ottimaCnustar}
C^{y,\hat{\nu}}(t)= C^{(a)}(t) \equiv C^{0}(t)\bigg( y \vee \sup_{0 \leq u \leq t}\bigg (\frac{a}{C^{0}(u)}\bigg )\bigg).
\eeq
\end{proposition}
\begin{proof}
We make the \textsl{ansatz} that $l^{*}(t)\equiv a$ for all $t \geq 0$ and we plug it into the left-hand side of (\ref{representationproblem}) to obtain
\begin{eqnarray}
\label{costCobb}
\lefteqn{a^{\alpha-1}\,\mathbb{E}\bigg\{\,\int_{\tau}^{\infty} e^{-\mu_Fs} \frac{C^{0}(s)}{C^{0}(\tau)}\,\bigg[\sup_{\tau \leq u \leq s}\bigg (\frac{C^{0}(s)}{C^{0}(u)}\bigg )\bigg]^{\alpha-1}\,ds\,\Big|\,\mathcal{F}_{\tau}\,\bigg\}} \nonumber \\
&&=a^{\alpha-1}\,\mathbb{E}\bigg\{\,\int_{\tau}^{\infty} e^{-\mu_Fs} \frac{C^{0}(s)}{C^{0}(\tau)}\,\inf_{\tau \leq u \leq s}\bigg ( \Big[\frac{C^{0}(s)}{C^{0}(u)}\Big]^{\alpha-1}\bigg )\,ds\,\Big|\,\mathcal{F}_{\tau}\,\bigg\}\nonumber \\
&&=a^{\alpha-1}\,\mathbb{E}\bigg\{\,\int_{\tau}^{\infty} e^{-\mu_Fs} e^{\sigma_C\left(W(s) -W(\tau)\right) - \widetilde{\mu}_C(s-\tau)}\,\\
&&\hspace{2.5cm}\times \inf_{0 \leq u' \leq s-\tau} \bigg[ e^{\sigma_C\left(W(s) -W(u'+\tau)\right) - \widetilde{\mu}_C(s-u'-\tau)}\bigg]^{(\alpha-1)}\,ds\,\Big|\,\mathcal{F}_{\tau}\,\bigg\}  \nonumber  \\
&&=a^{\alpha-1}e^{-\mu_F\tau}\,\mathbb{E}\bigg\{\,\int_{0}^{\infty} e^{-\mu_Fv}e^{\sigma_C W(v) - \widetilde{\mu}_Cv}\,\inf_{0 \leq u' \leq v} \bigg( e^{(\alpha-1)\left(\sigma_C\left(W(v) -W(u')\right) - \widetilde{\mu}_C(v-u')\right)}\bigg)\,dv\,\bigg\}\nonumber
\end{eqnarray}
since the Brownian increments are independent of $\mathcal{F}_{\tau}$.

If we now define $Y(v):=\widetilde{\mu}_Cv - \sigma_C W(v)$,\, $\underline{Y}(v):=\inf_{0 \leq u' \leq v}Y(u')$ and $\overline{Y}(v):=\sup_{0 \leq u' \leq v}Y(u')$, then we have
\begin{eqnarray}
\label{costCobb2}
&& a^{\alpha-1}e^{-\mu_F\tau}\,\mathbb{E}\bigg\{\,\int_{0}^{\infty} e^{-\mu_Fv}e^{\sigma_C W(v) - \widetilde{\mu}_Cv}\,\inf_{0 \leq u' \leq v} \bigg( e^{(\alpha-1)\left(\sigma_C\left(W(v) -W(u')\right) - \widetilde{\mu}_C(v-u')\right)}\bigg)\,dv\,\bigg\}\nonumber \\
&&= a^{\alpha-1}e^{-\mu_F\tau}\,\mathbb{E}\bigg\{\,\int_{0}^{\infty} e^{-\mu_Fv} e^{-\alpha Y(v)}e^{(\alpha - 1)\overline{Y}(v)} dv \bigg \} \nonumber \\
&&= \frac{1}{\mu_F}a^{\alpha-1}e^{-\mu_F\tau}\,\mathbb{E}\bigg\{\,\int_{0}^{\infty} \mu_F\,e^{-\mu_Fv}e^{- \alpha \left(Y(v) - \overline{Y}(v)\right)}e^{-\overline{Y}(v)}\,dv\bigg\}    \\
&&= \frac{1}{\mu_F}a^{\alpha-1}e^{-\mu_F\tau}\,\mathbb{E}\left\{e^{-\alpha \left(Y(\tau(\mu_F)) - \overline{Y}(\tau(\mu_F))\right)}e^{-\,\overline{Y}(\tau(\mu_F))}\right\}, \nonumber
\end{eqnarray}
where $\tau(\mu_F)$ denotes an independent exponentially distributed random time.

Using the Excursion Theory for Levy processes (cf. \cite{Bertoin}), $Y - \overline{Y}$ is independent of $\overline{Y}$, and by the Duality Theorem, $Y - \overline{Y}$ has the same distribution as $\underline{Y}$. Hence from (\ref{costCobb2}) we obtain
\begin{eqnarray}
\label{costCobb3}
\lefteqn{
\frac{1}{\mu_F}a^{\alpha-1}e^{-\mu_F\tau}\,\mathbb{E}\left\{e^{-\alpha \left(Y(\tau(\mu_F)) - \overline{Y}(\tau(\mu_F))\right)}e^{-\overline{Y}(\tau(\mu_F))}\right\}} \nonumber \\
&&=\frac{1}{\mu_F}a^{\alpha-1}e^{-\mu_F\tau}\,\mathbb{E}\left\{e^{-\alpha \underline{Y}(\tau(\mu_F))}\right\}\mathbb{E}\left\{e^{\,-\overline{Y}(\tau(\mu_F))}\right\}.
\end{eqnarray}
It is well known that for a Brownian motion with drift
$$\mathbb{E}\left\{e^{z \overline{Y}(\tau(\mu_F))}\right\}=\frac{\beta_{+}}{\beta_{+} - z}\,\,\,\,\,\,\,\,\mbox{and}\,\,\,\,\,\,\,\,\mathbb{E}\left\{e^{z\underline{Y}(\tau(\mu_F))}\right\}=\frac{\beta_{-}}{\beta_{-} - z},$$
if $\beta_{+}$ and $\beta_{-}$ are, respectively, the positive and negative roots of $\frac{1}{2}\sigma_C^2 x^2 + \widetilde{\mu}_Cx - \mu_F =0$, i.e.
$$\beta_{\pm}= -\frac{\widetilde{\mu}_C}{\sigma_C^2} \pm \sqrt{\bigg(\frac{\widetilde{\mu}_C}{\sigma_C^2}\bigg)^2 + \frac{2\mu_F}{\sigma_C^2}}.$$
Hence (cf.\ (\ref{representationproblem}))
\begin{eqnarray}
\label{costCobb4}
e^{-\mu_F \tau}&\hspace{-0.25cm} = \hspace{-0.25cm}&\mathbb{E}\bigg\{\,\int_{\tau}^{\infty} e^{-\mu_Fs} \frac{C^{0}(s)}{C^{0}(\tau)}\,\bigg[C^{0}(s) \sup_{\tau \leq u \leq s}\bigg (\frac{l^{*}(u)}{C^{0}(u)}\bigg )\bigg]^{\alpha-1}\,ds\,\Big|\,\mathcal{F}_{\tau}\,\bigg\} \nonumber  \\
&\hspace{-0.25cm} = \hspace{-0.25cm}&\frac{1}{\mu_F}a^{\alpha-1}e^{-\mu_F\tau}\mathbb{E}\left\{e^{-\alpha \underline{Y}(\tau(\mu_F))}\right\}\mathbb{E}\left\{e^{\,-\overline{Y}(\tau(\mu_F))}\right\} \\
&\hspace{-0.25cm} = \hspace{-0.25cm}&\frac{1}{\mu_F}a^{\alpha-1}e^{-\mu_F\tau}\frac{\beta_{+}\beta_{-}}{(1 + \beta_{+})(\alpha + \beta_{-})}.\nonumber
\end{eqnarray}

Then, we solve for $a$ and we obtain
$$a^{\alpha-1} = \bigg(\frac{\mu_F(1 + \beta_{+})(\alpha + \beta_{-})}{\beta_{+}\beta_{-}}\bigg),$$
which may also be written as
$$a = \bigg(\frac{2}{2\mu_F - \sigma_C^2\beta_{-} - \alpha\sigma_C^2(1 + \beta_{+})}\bigg)^{\frac{1}{1 - \alpha}}$$
being $\beta_{+}\beta_{-}= - \frac{2\mu_F}{\sigma_C^2}$.

Hence (cf.\ Theorem \ref{ottimasol}) the optimal capacity is
\beq
\label{optimalCCobbDouglas}
C^{y,\hat{\nu}}(t)=C^{(a)}(t)= C^{0}(t)\bigg( y \vee \sup_{0 \leq u \leq t}\bigg (\frac{a}{C^{0}(u)}\bigg )\bigg).
\eeq
\end{proof}

From Remark \ref{trovonubarrato} we have
\beq
\label{nustarCD2}
\overline{\nu}^{y}(t) = \sup_{0 \leq u \leq t}\bigg(\frac{a - yC^{0}(u)}{C^{0}(u)}\bigg) \vee 0,
\eeq
and the corresponding control $\hat{\nu}(t)$ (cf.\ (\ref{nubarradefinizione})) makes the diffusion reflect at the boundary $a$, it is the local time of $C^{y,\hat{\nu}}(t)$ at $a$.

Notice that the boundary $a$ in (\ref{lCobbDouglas}) coincides with the free boundary $k_b$ obtained via a viscosity solution approach by H.\ Pham in \cite{Pham} for a unit cost of investment $p$.
In fact from \cite{Pham}, Example $1.5.1$
$$k_b^{\alpha-1}= \frac{1-m}{C(\alpha -m)},$$
with
$$C=\frac{1}{\mu_F +\alpha\widetilde{\mu}_C - \frac{\alpha^2 \sigma_C^2 }{2}}\,\,\,\,\,\,\,\,\,\,\mbox{and}\,\,\,\,\,\,\,\,\,\,\,\,\,\,m = -\beta_{+},$$
and it is easy to see that
\beq
\label{akb}
a^{\alpha-1}=\frac{\mu_F(1 + \beta_{+})(\alpha + \beta_{-})}{\beta_{+}\beta_{-}}=\frac{1-m}{C(\alpha -m)}=k_b^{\alpha-1},
\eeq
hence $a=k_b$.

\begin{remark}
For a general production function $R(\cdot)$ satisfying Assumption \ref{AssProfit}, to find the free boundary $a$ one should solve the analogue of (\ref{costCobb3}), i.e.
$$\frac{1}{\mu_F}\mathbb{E}\left\{e^{-\underline{Y}(\tau(\mu_F))} R_{c}\left( a \, e^{-\underline{Y}(\tau(\mu_F))}\right) \right\}\mathbb{E}\left\{e^{-\overline{Y}(\tau(\mu_F))}\right\} =1,$$
or equivalently
$$\frac{1}{\mu_F}\mathbb{E}\left\{e^{-\underline{Y}(\tau(\mu_F))} R_{c}\left( a \, e^{-\underline{Y}(\tau(\mu_F))}\right) \right\} \frac{\beta_{+}}{1 + \beta_{+}} = 1.$$
That is, $a$ is the unique solution of
\beq
\label{remarkcostCD2}
\mathbb{E}\left\{e^{-\underline{Y}(\tau(\mu_F))} R_{c}\left( a \, e^{-\underline{Y}(\tau(\mu_F))}\right) \right\} = \frac{\mu_F (1 +\beta_{+})}{\beta_{+}}.
\eeq
Since now $-\underline{Y}(\tau(\mu_F)) = \overline{(-Y)}(\tau(\mu_F))$, and $\overline{(-Y)}(\tau(\mu_F))$ has exponential distribution of parameter $\gamma_{+}:=\frac{\widetilde{\mu}_C + \sqrt{\widetilde{\mu}_C^2 + 2 \mu_F\sigma^2_C}}{\sigma^2_C}>1$ (see, e.g., \cite{Bertoin}, Chapter VII), then (\ref{remarkcostCD2}) may be rewritten as
$$\int_{0}^{\infty}\gamma_{+} e^{x(1 - \gamma_{+})}R_c(a e^x) dx = \frac{\mu_F (1 +\beta_{+})}{\beta_{+}}.$$
\end{remark}


\appendix

\section{The Variational Approach in the Case of Time-Dependent \\ Coefficients}
\label{solutionInvestmentProblem}
\renewcommand{\theequation}{A-\arabic{equation}}

In this Appendix we revisit the solution of problem (\ref{optimalproblem}) obtained in Chiarolla and Haussmann \cite{Chiarolla2} by a variational approach and we generalize some of their results to the case of deterministic, time-dependent coefficients of the controlled diffusion (cf.\ Assumption \ref{detrminsticcoeff}).

Denote by $C^{t,y,\nu}(s)$ the capacity process starting at time $t \in [0,T)$ from $y$, controlled by $\nu$, with dynamics
\beq
\label{capacityimb}
\left\{
\begin{array}{ll}
dC^{t,y,\nu}(s)= C^{t,y,\nu}(s)[ -\mu_C(s) ds + \sigma_C(s) dW(s)] + f_C(s)d\nu(s),\,\,\,s\in[t,T), \\ \\
C^{t,y,\nu}(t)=y > 0.
\end{array}
\right.
\eeq
Hence
$$C^{t,y,\nu}(s)=\frac{C^{0}(s)}{C^{0}(t)}\bigg\{y + \int_{[t,s)}\frac{C^{0}(t)}{f_C(u)C^{0}(u)}d\nu(u)\bigg\}$$
with $C^{0}$ as defined in (\ref{GBM}).

To simplify notation write
\beq
\label{tildeC}
\widetilde{C}^{t}(s):=C^{t,1,0}(s)=\frac{C^{0}(s)}{C^{0}(t)}=e^{-\int_{t}^s(\mu_C(u) + \frac{1}{2}\sigma_C^2(u))du  + \int_{t}^s \sigma_C(u) dW(u)}, \quad Y^{t,y}(s):=y\widetilde{C}^{t}(s),
\eeq
and note that these processes are $\widetilde{\mathcal{F}}_{t,s}:=\sigma\{W(u)-W(t), t \leq u \leq s\}$-measurable.

To $C^{t,y,\nu}$ we associate the expected total profit, net of investment costs, given by
\beq
\label{profitimbedded}
J_{t,y}(\nu)=\mathbb{E}\bigg\{\int_t^{T} e^{-\int_t^s \mu_F(u) du}\,R(C^{t,y,\nu}(s))ds - \int_{[t,T)} e^{-\int_t^s \mu_F(u)du} d\nu(s) \bigg\},
\eeq
and the corresponding optimal investment problem is
\beq
\label{optimalproblemimbedded}
V(t,y):=\sup_{\nu \in \mathcal{S}_t}J_{t,y}(\nu),
\eeq
where
\begin{eqnarray*}
\mathcal{S}_t\hspace{-0.25cm}&:=&\hspace{-0.25cm}\{\nu:\Omega \times [t,T] \mapsto  \mathbb{R}_{+}\,\,\mbox{nondecreasing,\,\,left-continuous,\,\,adapted\,\,s.t.}\,\, \nu(t)=0,\,\,\mathbb{P}\mbox{-a.s.}\}
\end{eqnarray*}
is the convex set of irreversible investments.

We define the opportunity cost of not investing until time $s$ as (compare with \cite{Chiarolla2}, Section $3$)
\beq
\label{zetaimbedded}
\zeta^{t,y,T}(s):= \int_t^{s} e^{-\int_t^u \mu_F(r) dr }\,\widetilde{C}^t(u) R_c(y\widetilde{C}^t(u))du + e^{-\int_t^s \mu_F(r)dr}\widetilde{C}^t(s)\frac{1}{f_C(s)}\mathds{1}_{\{s < T\}},
\eeq
and the optimal stopping problem (compare with \cite{Chiarolla2}, eq.\ ($3.1$))
\beq
\label{optstoppprimbedded}
Z^{t,y,T}(s):= \essinf_{s \leq \tau \leq T}\mathbb{E}\left\{\zeta^{t,y,T}(\tau)\big|\widetilde{\mathcal{F}}_{t,s}\right\}.
\eeq
Denoting by $\mathcal{Z}^{t,y,T}(\cdot)$ the right-continuous with left-limits modification of $Z^{t,y,T}(\cdot)$, for $s=t$ we set $v(t,y):=\mathcal{Z}^{t,y,T}(t)$, so that up to a null set,
\begin{eqnarray}
\label{vimbedded}
v(t,y) \hspace{-0.25cm}&=&\hspace{-0.25cm} \inf_{t \leq \tau \leq T}\mathbb{E}\bigg\{\,\int_t^{\tau} e^{-\int_{t}^{u} \mu_F(r)dr}\,\widetilde{C}^t(u) R_c\left(y\widetilde{C}^t(u)\right)du \nonumber \\
&&\hspace{3.5cm}+\, e^{-\int_{t}^{\tau}\mu_F(r)dr}\widetilde{C}^t(\tau)\frac{1}{f_C(\tau)}\mathds{1}_{\{\tau < T\}}\,\bigg\}.
\end{eqnarray}
Now, the results in \cite{KaratzasBaldursson}, Proposition $2$ and Theorem $3$, guarantee that for $t \in [0,T)$ the stopping time
\beq
\label{taustarimbedded}
\tau^{*}(t,y)=\inf\left\{s \in [t,T): \mathcal{Z}^{t,y,T}(s)=\zeta^{t,y,T}(s)\right\} \wedge T
\eeq
is optimal for (\ref{optstoppprimbedded}), its left-continuous inverse (modulo a shift)
\beq
\label{nubarratoimbedded}
\overline{\nu}^{y}(s-t)=[\sup\left\{z \geq y: \tau^{*}(t,y) < s-t\right\} -y ]^{+}, \qquad s \in [t,T),
\eeq
is related to the optimal control $\hat{\nu}$ through $\hat{\nu}(s-t)=\int_{[t,s)}\frac{C^0(u-t)}{f_C(u-t)}d\overline{\nu}^{y}(u-t)$ (cf.\ Remark \ref{trovonubarrato}), and the function $v(t,y)$ is the shadow value of installed capital, i.e.
$$v(t,y)=\frac{\partial}{\partial y}V(t,y).$$

\begin{theorem}
\label{taustarvimbeddedThm}
Under Assumption \ref{detrminsticcoeff}, for every $(t,y)$ in $[0,T) \times (0,\infty)$ the optimal stopping time (\ref{taustarimbedded}) may be written as
\beq
\label{taustarvimbedded}
\tau^{*}(t,y)=\inf\bigg\{s \in [t,T): v(s,Y^{t,y}(s))=\frac{1}{f_C(s)}\bigg\} \wedge T,
\eeq
with $v$ as in (\ref{vimbedded}) and $Y^{t,y}$ as in (\ref{tildeC}).
\end{theorem}
\begin{proof}
Recall that $Y^{t,y}(s)=y\widetilde{C}^t(s)$ and (\ref{zetaimbedded}). Then, from (\ref{optstoppprimbedded}) we may write
\begin{eqnarray}
\label{taustarvimbedded1}
\mathcal{Z}^{t,y,T}(s) &\hspace{-0.25cm} = \hspace{-0.25cm}& \essinf_{s \leq \tau \leq T} \mathbb{E}\bigg\{\int_t^s e^{-\int_t^u \mu_F(r) dr }\,\widetilde{C}^t(u) R_c\left(Y^{t,y}(u)\right)du  \nonumber \\
&&\hspace{2.7cm} +   \int_s^{\tau} e^{-\int_t^u \mu_F(r) dr }\,\widetilde{C}^t(u) R_c\left(Y^{t,y}(u)\right)du \nonumber \\
&&\hspace{2.7cm} + e^{-\int_t^s \mu_F(r)dr}e^{-\int_s^{\tau} \mu_F(r)dr}\widetilde{C}^t(\tau)\frac{1}{f_C(\tau)}\mathds{1}_{\{\tau < T\}} \Big|\widetilde{\mathcal{F}}_{t,s}\bigg\}  \\
& \hspace{-0.25cm} = \hspace{-0.25cm} &\zeta^{t,y,T}(s) + \essinf_{s \leq \tau \leq T} \mathbb{E}\bigg\{\int_s^{\tau} e^{-\int_t^u \mu_F(r) dr }\,\widetilde{C}^t(u) R_c\left(Y^{t,y}(u)\right)du  \nonumber \\
&& \hspace{0.6cm} + e^{-\int_t^s \mu_F(r)dr}\Big(e^{-\int_s^{\tau} \mu_F(r)dr}\widetilde{C}^t(\tau)\frac{1}{f_C(\tau)}\mathds{1}_{\{\tau < T\}} - \widetilde{C}^t(s)\frac{1}{f_C(s)}\mathds{1}_{\{s < T\}}\Big)\Big|\widetilde{\mathcal{F}}_{t,s}\bigg\}. \nonumber
\end{eqnarray}
Notice now that
$$\widetilde{C}^t(u)=\widetilde{C}^t(s)\widetilde{C}^s(u),\,\,\,\forall u \geq t,\,\,\,\,\,\,\mbox{and}\,\,\,\,\,\,e^{-\int_t^{\tau} \mu_F(r)dr}\widetilde{C}^t(\tau)= e^{-\int_t^{s} \mu_F(r)dr}e^{-\int_s^{\tau} \mu_F(r)dr}\widetilde{C}^t(s)\widetilde{C}^s(\tau).$$
Hence for $s < T$ we have
\begin{eqnarray}
\label{taustarvimbedded2}
\mathcal{Z}^{t,y,T}(s) &\hspace{-0.25cm} = \hspace{-0.25cm}& \zeta^{t,y,T}(s)  \\
&& + e^{-\int_t^{s} \mu_F(r)dr}\widetilde{C}^t(s)\essinf_{s \leq \tau \leq T} \mathbb{E}\bigg\{\int_s^{\tau} e^{-\int_s^u \mu_F(r) dr }\,\widetilde{C}^s(u) R_c\left(Y^{t,y}(s)\widetilde{C}^s(u)\right)du  \nonumber \\
&& \hspace{4.7cm} + e^{-\int_s^{\tau} \mu_F(r)dr}\widetilde{C}^t(\tau)\frac{1}{f_C(\tau)}\mathds{1}_{\{\tau < T\}} - \frac{1}{f_C(s)} \Big|\widetilde{\mathcal{F}}_{t,s}\bigg\}. \nonumber
\end{eqnarray}
In order to take care of the conditioning in (\ref{taustarvimbedded2}) we proceed exactly as in the proof of Proposition \ref{propYtildevty} and recalling (\ref{vimbedded}), from (\ref{taustarvimbedded2}) we get
\beq
\label{taustarvimbedded3}
\mathcal{Z}^{t,y,T}(s) = \zeta^{t,y,T}(s) + e^{-\int_t^{s} \mu_F(r)dr}\widetilde{C}^t(s)\left(v(s,Y^{t,y}(s))-\frac{1}{f_C(s)}\right).
\eeq
Finally, (\ref{taustarimbedded}) and (\ref{taustarvimbedded3}) imply
\begin{eqnarray}
\tau^{*}(t,y)&\hspace{-0.25cm} = \hspace{-0.25cm}&\inf\{s \in [t,T): \mathcal{Z}^{t,y,T}(s)=\zeta^{t,y,T}(s)\} \wedge T \nonumber \\
&\hspace{-0.25cm} = \hspace{-0.25cm}&\inf\bigg\{s \in [t,T): e^{-\int_t^{s} \mu_F(r)dr}\widetilde{C}^t(s)\left(v(s,Y^{t,y}(s))-\frac{1}{f_C(s)}\right)= 0\bigg\} \wedge T   \nonumber \\
&\hspace{-0.25cm} = \hspace{-0.25cm}&\inf\bigg\{s \in [t,T): v(s,Y^{t,y}(s))=\frac{1}{f_C(s)}\bigg\} \wedge T.
\end{eqnarray}
\end{proof}

Notice that if $\widetilde{\mathbb{E}}\left \{\cdot \right\}$ is the expectation w.r.t.\ $\widetilde{\mathbb{P}}$ (cf.\ (\ref{RADON}) for its definition), then by Girsanov Theorem (\ref{vimbedded}) may also be written as
\begin{eqnarray}
\label{vtyptilde}
v(t,y)&\hspace{-0.25cm}  = \hspace{-0.25cm} & \inf_{t \leq \tau \leq T}\widetilde{\mathbb{E}}\bigg\{\,\int_t^{\tau} e^{-\int_t^u \overline{\mu}(r)dr}\,R_c\left(Y^{t,y}(u)\right)du + e^{-\int_t^{\tau}\overline{\mu}(r)dr}\frac{1}{f_C(\tau)}\mathds{1}_{\{\tau < T\}}\,\bigg\}\nonumber 
\end{eqnarray}
with $\overline{\mu}(t)= \mu_F(t) + \mu_C(t)$, for $t \in [0,T]$, and where the second equality is due to (\ref{tildeC}).
The value function $v(t,y)$ is expected to be the solution of a variational inequality similar to that obtained in Chiarolla and Haussmann \cite{Chiarolla2} under Markovian restrictions (cf.\ \cite{Chiarolla2}, Assumption-[M]) and with a Cobb-Douglas production function (cf.\ \cite{Chiarolla2}, eq.\ $(4.5)$ and Theorem $4.4$).

\bigskip

\textbf{Acknowledgments.} The authors thankfully acknowledge two anonymous referees for their pertinent and useful comments.


\end{document}